\documentclass[12pt]{amsart}
\usepackage{amsfonts,amssymb, amscd, latexsym, graphicx, psfrag, color, float}
\usepackage[all]{xy}

\usepackage{todonotes}
\usepackage{hyperref}
\usepackage{mathrsfs}
\usepackage{mathtools}

\newtheorem{theorem}{Theorem}[section]
\newtheorem*{theoremnn}{Theorem}
\newtheorem{proposition}[theorem]{Proposition}
\newtheorem{proposition-definition}[theorem]
{Proposition-Definition}
\newtheorem{corollary}[theorem]{Corollary}
\newtheorem{lemma}[theorem]{Lemma}

\theoremstyle{definition}
\newtheorem*{notationnn}{Notation}
\newtheorem{definition}[theorem]{Definition}
\newtheorem{notation}[theorem]{Notation}
\newtheorem{example}[theorem]{Example}

\newtheorem{remark}[theorem]{Remark}

\theoremstyle{remark}

\newcommand{\bA}{\mathbb{A}}
\newcommand{\bC}{\mathbb{C}}

\newcommand{\bN}{\mathbb{N}}

\newcommand{\bQ}{\mathbb{Q}}
\newcommand{\bR}{\mathbb{R}}
\newcommand{\bZ}{\mathbb{Z}}

\newcommand{\cO}{\mathcal{O}}

\newcommand{\cT}{\mathcal{T}}
\newcommand{\cU}{\mathcal{U}}

\newcommand{\cX}{\mathcal{X}}
\newcommand{\cY}{\mathcal{Y}}
\newcommand{\cZ}{\mathcal{Z}}

\newcommand{\St}{\mathrm{St}}

\newcommand{\Spec}{\mathrm{Spec}\,}

\newcommand{\Hom}{\mathrm{Hom}}

\newcommand{\Div}{\mathrm{Div}}

\renewcommand{\log}{\mathrm{log}}

\newcommand{\gp}{\mathrm{gp}}
\newcommand{\an}{\mathrm{an}}

\newcommand{\et}{\mathrm{\acute{e}t}}

\newcommand{\chara}{{\rm char}\,}

\newcommand{\ket}{\mathrm{k\acute et}}
\newcommand{\fket}{\mathrm{fk\acute et}}
\newcommand{\fet}{\mathrm{f\acute et}}

\begin{document}

\title{The log homotopy exact sequence}
\author{Mattia Talpo}
\address{Mattia Talpo\\ Department of Mathematics\\
University of Pisa\\
Largo Bruno Pontecorvo 5\\
56127 Pisa PI\\
Italy}
\email{mattia.talpo@unipi.it}


\maketitle

\begin{abstract}
We show exactness of the homotopy sequence for the logarithmic fundamental group in the case of log smooth, finitely presented, proper and saturated morphisms of fs log schemes over a field. This generalizes earlier results of Hoshi \cite{hoshi} in the log regular case. In passing, we also construct a ``log Stein factorization'' in some particular cases. \end{abstract}

\setcounter{tocdepth}{1}

\tableofcontents

\section{Introduction}

This paper is about an exact sequence of fundamental groups in the context of certain log smooth morphisms of log schemes.

Recall that in topology, a fibration of path connected pointed topological spaces $(X,x)\to (Y,y)$ induces a long exact sequence of homotopy groups
$$
\xymatrix{
\cdots \ar[r] & \pi_{i+1}(Y,y)\ar[r] & \pi_i(X_y,x)\ar[r] & \pi_i(X,x)\ar[r] & \pi_i(Y,y)\ar[r]& \pi_{i-1}(X_y,x)\ar[r] & \cdots
}
$$
relating the homotopy groups of $X$, of $Y$, and of the fiber $X_y$ over $y$. There is an analogue of this in algebraic geometry for certain kinds of fibrations, involving the \emph{\'etale homotopy groups} $\pi_i^\et(X,x)$ of a connected scheme $X$ equipped with a geometric point $x$, that are defined in terms of the \emph{\'etale homotopy type} \cite{ArtinMazur, Friedlander1, Friedlander2}. A special and simpler case of this is given by an exact sequence for the \'etale fundamental groups, that was proven by Grothendieck \cite[Corollaire X.1.4]{SGA1}, and that holds for a proper, separable morphism of connected locally noetherian schemes with geometrically connected fibers.

Logarithmic geometry is a variant of algebraic geometry, where schemes are equipped with a structure sheaf of monoids with a multiplicative monoid homomorphism to the structure sheaf. The theory was first developed by Kato, Fontaine, Illusie, Deligne, Faltings, and then many others after them. Although it was initially engineered to be applied to some questions in $p$-adic Hodge theory, in time it became very relevant for the study of moduli spaces and enumerative geometry. In log geometry one has appropriate variants of the classical concepts of algebraic geometry, and this includes a notion of \'etale fundamental groups, which encode so-called \emph{Kummer \'etale covers} of a log scheme, the appropriate notion of finite covering spaces in this context. One can also define higher homotopy groups via an analogue of the \'etale homotopy type of a scheme (see for example \cite{KNIRS, proflog}). It is natural to expect exact sequences for these higher homotopy groups for certain morphisms, in presence at least of log smoothness and properness, in analogy with the topological and algebro-geometric cases. This is further hinted at by the log version of Ehresmann's lemma, i.e. the fact that over $\bC$, for a log smooth, proper and exact morphisms $X\to Y$, the induced map between the Kato-Nakayama spaces (a version of the analytification functor for log schemes) $X_\log\to Y_\log$ is a fiber bundle.

The problem of exactness of the sequence of log fundamental groups was considered by Hoshi in \cite{hoshi}, where he proved it assuming that the base log scheme is log regular. This is of course a quite special case, insufficient for applications to moduli problems where one wants to allow arbitrary fs base log schemes, for example non-trivial log points. Similar questions (but using a different notion of fundamental group) are considered also in \cite{dipr}.

We now state the main theorem of this paper. Let $k$ be a field.

\begin{theoremnn}[Theorems \ref{theorem:main} and \ref{thm:left.exact}]
Let $f\colon X\to S$ be a proper, finitely presented, log smooth and saturated morphism of fs log schemes over $k$, with connected log geometric fibers. Let $x\to X$ be a log geometric point of $X$, and $s\to S$ be its image in $S$. Then there is an induced exact sequence of log fundamental groups
$$
\xymatrix{
\pi_1^\ket(X_{{s}},{x})\ar[r] & \pi_1^\ket(X,{x})\ar[r] & \pi_1^\ket(S,{s})\ar[r] & 1.
}
$$
Moreover, if $S$ is a log point, we can remove the log smoothness and saturatedness assumptions and the sequence is also exact on the left.
\end{theoremnn}

Here $X_s$ is the ``log gometric fiber'' of the morphism $f$ over the log geometric point $s\to S$. It is a saturated log scheme, but the log structure is not finitely generated. We point out that one can avoid the use of such non-finitely generated log structures by considering appropriate limits of fs log schemes (as for example in \cite{hoshi}), but we feel that doing so makes statements and reasoning more cumbersome, and using this log geometric fiber is more natural. Non-finitely generated log structures have come up quite naturally in several contexts recently (cfr. for example \cite{MolchoWise, beyondfs}). Note also that we have a smoothness assumption on the morphism, which is not present in Grothendieck's theorem. It is likely that this assumption can be relaxed, but it is important for our proof, and it is going to be satisfied for the applications that we envision at the moment. 

The proof of the theorem is going to be an adaptation to the logarithmic case of existing proofs for the case of schemes (as in \cite[\href{https://stacks.math.columbia.edu/tag/0BUM}{Tag 0BUM}]{stacks-project}). One of the ingredients (just as in Hoshi's paper) is the Stein factorization, that we construct in the logarithmic context under some assumptions (Proposition \ref{prop:stein.factorization}). In order to mimic the non-log proof, we make use of a correspondence between Kummer \'etale covers of a log scheme $X$ and classically finite \'etale covers of some root stack of $X$. This was mostly worked out by Olsson in \cite{Olsson}. His construction, complemented with some input from the theory of infinite root stacks \cite{irs}, gives an equivalence of categories between the aforementioned objects, that allows us to transfer the problems on the logarithmic side to questions on root stacks.

The results of this paper are going to be useful (and were in fact motivated by) an upcoming work with Martin Ulirsch and Dmitry Zakharov about moduli spaces of log and tropical curves with level structures.

\subsection*{Conventions}

We work over a base field $k$. All of our monoids will be commutative and with a neutral element. We denote by $\bQ$ the set of rational numbers with denominator prime to $\chara k$. If $P$ is a fine saturated monoid, $P_\bQ$ will denote the cone determined by $P$ in the vector space $P^\gp\otimes_\bZ\bQ$, which can also be thought of as the direct limit $\varinjlim_n \frac{1}{n}P$ where $n$ ranges along the natural numbers not divisible by $p$, and the order is given by divisibility.

All log schemes will be (integral and) saturated, but not always fine, see Section \ref{sec:log.schemes} for a brief discussion. 

If $f\colon X\to Y$ is a morphism of (log) schemes and $F$ is a sheaf on $Y$, if there is no risk of confusion we will sometimes denote by $F|_X$ the pullback (in the appropriate category) of $F$ to $X$ along $f$.

Sometimes we write ``$=$'' for a canonical isomorphism, which we trust that the reader is able to determine easily, in case it isn't specified in the text.

\subsection*{Acknowledgements}

The author warmly thanks Piotr Achinger, Matthieu Romagny, David Rydh, Angelo Vistoli for useful conversations on the topics of this paper. He is also a grateful member of INdAM.

\section{Preliminaries}

\subsection{Log schemes}\label{sec:log.schemes}

We will assume that the reader is familiar with the basics of log geometry, and we only recall some relevant definitions and results. We refer to \cite{ogus-book} for basic notions and results that we do not explicitly recall here. If $X$ is a log scheme, we will denote by $\underline{X}$ the underlying scheme, as is often customary.

The most widely used category of log schemes is without a doubt the one of fs, for \emph{fine and staturated}, log schemes. These are log schemes $X$ that \'etale locally admit charts $X\to \Spec(P\to A)$, with $P$ a fs monoid, i.e. integral finitely generated (this is the ``fine'' bit) and saturated. Here, if $A$ is commutative unital ring, $P$ is a monoid and $P\to A$ is a homomorphism, where $A$ is regarded as a monoid via multiplication, with the symbol $\Spec(P\to A)$ we denote the affine log scheme $\Spec A$, equipped with the log structure induced by the chart $P\to A$.

To work with log geometric fibers (see Section \ref{sec:log.g.fibers}) we need to sometimes enlarge this category by dropping the ``finitely generated'' assumption, and work with saturated, but not necessarily fs, log schemes. Our work is not the first instance where the ``finitely generated'' assumption is too restrictive. As examples we point out the work of Molcho--Wise on the log Picard group \cite{MolchoWise}, which make use of non-finitely generated valuative monoids, and the paper \cite{beyondfs}, which develops parts of the basic theory of log schemes (including the log fundamental group) in absence of finite generation. The fs assumption will still be often present in this paper, except when we will be dealing with log geometric fibers.

Note that by \cite[Lemma 3.1.2]{beyondfs}, $X$ is saturated if and only if it is quasi-coherent (i.e. it admits charts \'etale locally) and all geometric stalks of $M_X$ are saturated monoids.

\begin{example}
A typical example of a non-finitely generated saturated monoid that we care about would be $\bQ_+$, the monoid of non-negative rational numbers with respect to addition, or more generally the monoid of rational point $P_\bQ$ of the cone generated in $P^\gp\otimes_\bZ \bR$ by the fs monoid $P$. These are monoids that one encounters when considering log geometric points, and  consequently log geometric fibers of morphisms of fs log schemes.
\end{example}

We now recall the notions of integral and saturated morphisms, which will be important throughout the text.

\begin{definition}
A homomorphism $P\to Q$ of integral monoids is called \emph{integral} if the coproduct $Q\oplus_P P'$ along any morphism $P\to P'$ with $P'$ integral, is again an integral monoid. A morphism $f\colon X\to Y$ of integral log schemes is called integral if for every geometric point ${x}\to X$ the morphism $M_{Y,f({x})}\to M_{X,x}$ is integral.
\end{definition}

Integrality of a homomorphism of integral monoids $\phi\colon P\to Q$ is equivalent to flatness of the induced morphism $\Spec \bZ[Q]\to \Spec \bZ[P]$, and also to an ``equational'' definition, namely the fact that whenever we have an equality $\phi(p_1)+q_1=\phi(p_2)+q_2$ in $Q$, there exist $q'\in Q$ and $p_1', p_2'$ in $P$ such that $p_1+p_1'=p_2+p_2'$ and $q_i=\phi(p_i')+q'$, see \cite[Proposition I.4.6.7]{ogus-book}.

We recall in particular that a log smooth morphism which is integral is also flat (in the classical sense), see \cite[Theorem IV.4.3.5]{ogus-book}.

\begin{definition}
A homomorphism $ P\to Q$ of saturated monoids is called \emph{saturated} if the coproduct $Q\oplus_P P'$ along any morphism $P\to P'$ with $P'$ saturated, is again a saturated monoid. A morphism $f\colon X\to Y$ of integral log schemes is called saturated if for every geometric point ${x}\to X$ the morphism $M_{Y,f({x})}\to M_{X,x}$ is saturated.
\end{definition}  
 
For log smooth and integral morphisms of fs log schemes, saturatedness is equivalent to reducedness of fibers, see \cite[Theorem IV.4.3.6]{ogus-book}.

We also briefly recall how fibered products of log schemes are constructed (see \cite[Section III.2.1]{ogus-book}): given two morphisms $Y\to X$ and $Z\to X$,  one forms the fibered product of the underlying schemes $\underline{Y}\times_{\underline{X}}\underline{Z}$ and equips it with the pushout log structure coming from the obvious diagram. This gives the fibered product in log schemes, and if $X,Y,Z$ are coherent (i.e. locally admit charts from finitely generated monoids), then the fibered product is also coherent. It might well happen though that this fibered product is not integral (resp. saturated), even if $X,Y,Z$ are. This is true if one of the two morphisms is integral (resp. saturated).

If the naive fibered product is not fine (resp. saturated), one obtains the fibered product in the category of fine (resp. fs) log schemes by applying an ``integralization'' (resp. ``saturation'') functor, which is modelled on the integralization (resp. saturation) functors on monoids. If $P$ is any monoid, the integralization $P^{\rm int}$ is defined as the image of $P$ in the associated group $P^\gp$, and is the initial integral monoid with a morphism from $P$. The saturation $P^{\rm sat}$ of an integral monoid $P$ is given by the submonoid of $P^\gp$ of elements that have a positive multiple in $P$, and is the initial saturated monoid with a morphism from $P$. Both of these functors are left adjoint to the appropriate inclusion functors. The corresponding functors on log schemes are defined by pulling back this construction along charts for the log structure. For details, see \cite[Proposition III.2.1.5]{ogus-book}. The same constructions extend to log algebraic stacks with only cosmetic changes.

 This all works well with fine or fs log schemes or stacks. In the case of saturated log structures (without the finite generation assumption) one has to be more careful, and it is not clear whether fibered products always exist and behave well in full generality. See \cite[Lemma 3.1.9]{beyondfs} and the surrounding discussion.  
 
\begin{notationnn}
All fibered products of log schemes and stacks will tacitly be taken in the saturated category. In particular, whenever we write a fibered product of log schemes or stacks, we are implicitly asserting that it exists. In all cases, existence will follow from \cite[Lemma 3.1.9]{beyondfs}.
\end{notationnn}

To conclude this section we also recall the notion of exactness.

\begin{definition}
A homomorphism $P\to Q$ of monoids is \emph{exact} if the natural map $P\to Q\times_{Q^\gp} P^\gp$ is an isomorphism. A morphism $f\colon X\to Y$ of log schemes is exact if for every geometric point $x\to X$ the morphism $(f^{*}M_{Y})_x\to M_{X,x}$ is exact.
\end{definition}

If $P$ and $Q$ are integral, exactness simply means that $P$ is the preimage of $Q$ along the induced homomorphism $P^\gp\to Q^\gp$. Typical examples of non-exact morphisms are log blow-ups, for example the morphism $\bA^2\to \bA^2$ given by $(x,y)\mapsto (x,xy)$, which is a chart of the blow-up of the origin in $\bA^2$, and is a toric morphism induced by the monoid homomorphism $\bN^2\to \bN^2$ mapping $e_1$ to $e_1$ and $e_1$ to $e_1+e_2$.

Imposing exactness gives a log version of Ehresmann's theorem about locally trivial fibrations: if $f\colon X\to Y$ is a log smooth, proper and exact morphism of fs log schemes locally of finite type over $\bC$ (or log analytic spaces), then the induced morphism $X_\log\to Y_\log$ is a fiber bundle, i.e. locally a product over the base \cite[Theorem 0.3]{nakayama-ogus}. Here $X_\log$ is the \emph{Kato--Nakayama space} of $X$, first introduced in \cite{KN}, a topological space obtained from the analytification $X_\an$ by replacing points with non-trivial log structure with real tori.

We also point out that by \cite[Proposition III.2.5.2]{ogus-book}, a saturated morphism is exact.

\subsection{Root stacks}

We briefly recall the construction of (finite and infinite) root stacks of a log scheme. Our references for this are \cite{borne-vistoli} and \cite{irs}.

Let $X$ be a log scheme, equipped with a system of denominators $\overline{M}_X\to N$. This is an injective homomorphism of (\'etale) sheaves of monoids which is \emph{of Kummer type} on all stalks at geometric points (i.e. for every geometric point $x$ of $X$, every element of $N_x$ has a positive multiple coming from $\overline{M}_{X,x}$), and $N$ is coherent, i.e. it admits local charts from finitely generated monoids.

Recall that the log structure $\alpha\colon M_X\to \cO_X$ of $X$ can also be seen as a \emph{Deligne--Faltings structure} (DF structure for short), which is a symmetric monoidal functor $L\colon \overline{M}_X\to \Div_X$, where $\Div_X$ is the stack on the small \'etale site of $X$ consisting of line bundles together with a global section (sometimes these are called ``generalized Cartier divisors''). The DF structure is obtained from the log structure by taking the stacky quotient for the action of the units $\cO_X^\times$.

\begin{definition}
The \emph{root stack} $\sqrt[N]{X}$ of $X$ with respect to $\overline{M}_X\to N$ is the stack on schemes over $X$ whose $(S\to X)$-points are extensions
$$
\xymatrix{
\overline{M}_X|_S\ar[d] \ar[r] & \Div_S\\
N|_S\ar@{-->}[ru] &
}
$$
of the pullback DF structure $\overline{M}_X|_S\to \Div_S$ of $X$ to $S$ to the sheaf $N|_S$. An important special case is given by $N=\frac{1}{n}\overline{M}_X$ for some natural number $n$. In that case the root stack is simply denoted by $\sqrt[n]{X}$.
\end{definition}

The root stacks of a fine log scheme are tame algebraic stacks with finite diagonalizable stabilizers and coarse moduli space the original $X$, and they are also Deligne--Mumford when $\chara k$ does not divide the order of the cokernel of the morphism $\overline{M}_{X,x}^\gp\to N_x^\gp$, for any geometric point $x$ of $X$ (in particular this happens for $\sqrt[n]{X}$ if $\chara k\nmid n$), see \cite[Proposition 4.19]{borne-vistoli}.

\begin{remark}\label{remark:valued.in.sets}
A root stack $\sqrt[N]{X}$ of a log scheme $X$ is naturally equipped with a tautological log structure, with characteristic sheaf given by the pullback of the sheaf $N$ from $X$.

If $X$ is a log scheme and $\overline{M}_X\to N$ is a Kummer extension, the corresponding root stack $\sqrt[N]{X}$, seen as a log stack in this way, gives a groupoid-valued functor $F$ on the category of log schemes, by setting $F(T)=\Hom(T,\sqrt[N]{X})$ (where $\Hom$ denotes morphisms of log stacks). This functor is actually valued in sets.

The reason is that the presence of the morphism $ N|_T\to \overline{M}_T$ between the log structures of $T$ and $\sqrt[N]{X}$ and of the natural transformation witnessing commutativity of the diagram
$$
\xymatrix{
N|_T\ar[r]\ar[d] & \Div_T\\
\overline{M}_T \ar[ru]
}
$$
kills the automorphisms that the DF structure $N|_T\to \Div_T$ might have (and that give the stacky structure over schemes), since these automorphisms would have to be compatible with the diagram above.
\end{remark}

Root stacks form an inverse system with respect to the natural order relation on systems of denominators on $X$, given by inclusion. The inverse limit, that can be identified with the a priori smaller $\varprojlim_n \sqrt[n]{X}$, is called the \emph{infinite root stack} and denoted by $\sqrt[\infty]{X}$. This is not an algebraic stack anymore, but it still has an fpqc atlas and presentations as a quotient stack \'etale locally on the coarse moduli space $X$ \cite[Corollary 3.13]{irs}.

On top of this, the formation of root stacks is also functorial with respect to morphisms of log schemes equipped with systems of denominators: if $X$ and $Y$ are log schemes with systems of denominators $\overline{M}_X\to N$ and $\overline{M}_Y\to N'$, and we are given a morphism of log schemes  $X\to Y$ together with a homomorphism $N'|_X\to N$ of sheaves of monoids compatible with $\overline{M}_Y|_X\to \overline{M}_X$, then there is an induced morphism of root stacks $\sqrt[N]{X}\to \sqrt[N']{Y}$, defined by restriction along the given $N'|_X\to N$.

\begin{remark}
We will make use of root stacks to turn Kummer \'etale covers of a log scheme into strict \'etale covers of its root stacks (Section \ref{sec:correspondence}). Because of a condition on the homomorphisms of monoids corresponding to Kummer \'etale maps (that we recall in the next section), it suffices to consider the inverse limit of the root stacks only over the natural numbers not divisible by $\chara k$, which will be pro-Deligne--Mumford (an inverse limit of Deligne--Mumford stacks) instead of merely pro-algebraic. We will denote the corresponding ``prime-to-$\chara k$'' infinite root stack $\varprojlim_{\chara k \nmid n}\sqrt[n]{X}$ by $\sqrt[\infty]{X}'$. In characteristic zero this will of course be the full infinite root stack.
\end{remark}

We will at some point have to consider root stacks of log geometric fibers, which are going to be saturated log schemes, but not fs. The functorial definition of root stacks makes perfect sense in that case as well, and it is everything that we will need to carry out our arguments.

\subsection{The logarithmic fundamental group}

We now recall the definition and some facts about the logarithmic (or Kummer \'etale) fundamental group. We refer the reader to \cite[Section 4]{illusie-survey} for a more detailed account.

The log fundamental group is defined in the same spirit as the usual \'etale fundamental group of a connected scheme, as the automorphisms of the fiber functor with respect to a log geometric point, on a certain category of finite covers. It is a profinite group, inverse limit of automorphism groups of the Galois objects of the category.

We start by recalling the notion of log geometric point.

\begin{definition}
A \emph{log geometric point} of $X$ is a morphism $x\to X$, where $x=\Spec(\Omega \to K)$ is a saturated log scheme, with underlying scheme $\Spec K$ for a separably closed field $K$, and for which the monoid $\Omega$ is \emph{uniquely divisible}, i.e. it has the property that $\cdot n\colon \Omega\to \Omega$ is bijective for every $n$ invertible in $K$. The map $\Omega\to K$ sends $0$ to $1$, and every $c\in \Omega\setminus\{0\}$ to $0\in K$, as usual for log points.
\end{definition}

\begin{remark}
Log geometric points can be identified with points (in the sense of topos theory) of the log \'etale topos, in analogy with the classical case of the \'etale site \cite[Proposition 4.4]{illusie-survey}.
\end{remark}

A log geometric point $x\to X$ has an underlying geometric point of $\underline{X}$, that we denote by $\underline{x}\to \underline{X}$.

Going in the other direction, if $x=\Spec K \to \underline{X}$ is a geometric point of the underlying scheme one can construct a  log geometric lying point lying over it in the following way: given an isomorphism $M_x\cong K^\times\oplus P$, equip $x$ with the log structure given by the monoid $K^\times \oplus P_\bQ$, and homomorphism $K^\times\oplus P_\bQ\to K$ sending $P_\bQ$ to zero. This gives a log geometric point $x'$ over $x$ (of course by means of the natural map $P\to P_\bQ$), and a choosing a different isomorphism produces a canonically isomorphic log geometric point.

Moreover, the morphism $x'\to X$ has the property that the map $\overline{M}_x\to \overline{M}_{x'}$ induces an isomorphism $(\overline{M}_x)_\bQ\cong \overline{M}_{x'}$. 

\begin{definition}
A log geometric point $x=\Spec(\Omega \to K)\to X$  is \emph{strict} if the map $\overline{M}_x\to \Omega$ induces an isomorphism $(\overline{M}_x)_\bQ\cong \Omega$.
\end{definition} 

Notice that this doesn't mean that the morphism of log schemes is strict in the usual sense, and this differs from the use of the term in \cite{hoshi}. Over every geometric point of $X$ there is a unique (up to isomorphism) strict log geometric point, given by the construction above. There are of course plenty of non-strict log geometric points, obtained for example by enlarging the monoid $(\overline{M}_x)_\bQ$ further.

The finite covers in the logarithmic context are the \emph{Kummer \'etale covers} of a fs log scheme $X$. These are maps of log schemes $f\colon Y\to X$ that are log \'etale, locally of finite presentation (in the usual sense), and of Kummer type, i.e. for every geometric point $y$ of $Y$, the morphism $\overline{M}_{X,f(y)}\to \overline{M}_{Y,y}$ is of Kummer type (this is the same condition for the systems of denominators of the previous section). We recall that being log \'etale is defined by an infinitesimal lifting property along log thickenings \cite[Definition IV.3.1.1]{ogus-book}, and is also equivalent to \'etale locally having charts modelled on injective monoid homomorphisms $P\to Q$ with finite cokernel of order prime to $\chara k$ \cite[Theorem IV.3.3.1]{ogus-book}.

\begin{example}
Of course all strict \'etale finite covers are also Kummer \'etale covers. Prototypical examples of non-classically \'etale but Kummer \'etale covers are (tamely) ramified covers of the affine toric variety $\Spec k[P]\to \Spec k[P]$ for some fs monoid $P$, induced my multiplication by $n \colon P\to P$, with $\chara k \nmid n$. The simplest example is the $n$-th power map $\bA^1\to \bA^1$ sending $x$ to $x^n$.
\end{example}

Kummer \'etale covers enjoy many of the properties of usual \'etale covers, for example they can be split (i.e. they become isomorphic to a projection of the form $\bigsqcup_{i\in I} X\to X$ for some finite set $I$) by passing to some Kummer \'etale cover, and if $X\to Y\to Z$ and $Y\to Z$ are finite Kummer \'etale, then $X\to Y$ also is: the fact that it is log \'etale is an immediate consequence of definitions, and for the ``Kummer'' property - which isn't much harder - see \cite[Remark 6.13]{irs}

\begin{notation}
If $X$ is a scheme or an algebraic stack, we will denote by $\fet(X)$ the category of finite \'etale covers of $X$. Similarly if $X$ is a log scheme or log algebraic stack, we will denote by $\fket(X)$ the category of finite Kummer \'etale covers of $X$.

In the same spirit, we will often abbreviate ``finite \'etale'' with \emph{f\'et} and ``finite Kummer \'etale'' with \emph{fk\'et}.
\end{notation}

We will also have to consider the fundamental group of a log geometric fiber of a morphism between fs log schemes, which is not itself an fs log scheme. The theory of Kummer \'etale covers and the log fundamental group for saturated (but not necessarily fs) log schemes has been recently developed in \cite{beyondfs}, which we refer the reader to. Things work more or less in the same way as for fs log schemes, with the addition of the important \emph{sfp} condition, which stands for ``finitely presented up to saturation''. It is worthwhile to point out that finite Kummer \'etale covers of a saturated log schemes are not in general finite morphisms of schemes (as the name seems to suggest), but only integral (i.e. they are affine morphisms locally given by integral extensions of rings), see  \cite[Section 4.2]{beyondfs}. In any case, we will not have to use any deep fact about the development of the theory.

\subsection{Log geometric fibers}\label{sec:log.g.fibers}

In this section we recall the notion of log geometric fibers and prove some useful facts about them, especially regarding the interactions with root stacks.

\begin{definition}
Let $f\colon X\to Y$ be a morphism of fs log schemes, and ${y}\to Y$ a log geometric point. The \emph{log geometric fiber} of $f$ at $y$ is the saturated log scheme $X_{{y}}$ given by the fibered product $X\times_Y{y}$, in the category of saturated log schemes.
\end{definition}

Of course the log geometric fiber will typically not be a fs log scheme.

\begin{definition}\label{def:log.g.connected}
A morphism $f\colon X\to Y$ of fs log schemes is \emph{log geometrically connected} if for every strict log geometric point ${y}\to Y$, the log geometric fiber $X_{{y}}$ is connected.
\end{definition}

\begin{remark}
It is natural to wonder whether it would be equivalent to ask that the fiber over \emph{any} log geometric point be connected. However there is no a priori reason to expect this, since morphisms between log geometric points might not be integral, and therefore when taking the fibered product one might have to pass to a closed subscheme, and connectedness might be lost. If $f$ is integral, this issue certainly doesn't occur, and the stronger condition is equivalent to the weaker one that we are using.
\end{remark}

\begin{remark}
Definition \ref{def:log.g.connected} may look more natural, once one accepts to work with non finitely generated monoids, than \cite[Definition 1.7]{hoshi}, which requires all fibers over ``reduced covering points'' over the given geometric point to be connected. Since a strict log geometric point is the inverse limit of these reduced covering points (which can be seen the ``rigidified'' reduction of a root stack of the log point), our definition does coincide with Hoshi's at least when $f$ is quasi-compact and intergral (using \cite[Proposition 8.4.1]{EGA4-3}). \end{remark}

\begin{remark}\label{rmk:log.fibers.vs.fibers}
If a log geometrically connected morphism $X\to Y$ is assumed to be integral, then the underlying morphism of schemes is geometrically connected, as in that case the morphism $X_y\to \underline{X}_{\underline{y}}$ is surjective (the target is the geometric fiber of the underlying map of schemes $\underline{X}\to \underline{Y}$). If $X\to Y$ is also saturated, then the underlying scheme of the log geometric fibers coincide with the corresponding geometric fiber.
\end{remark}

For future use, we include the following results, that relate log geometric fibers of morphisms of log schemes to those of an induced morphism of root stacks.

\begin{lemma}\label{lemma:root.mono}
Any root stack $\cT\to T$ is a monomorphism on saturated log schemes, i.e. if $S$ is a saturated log scheme and $S\rightrightarrows \cT$ are morphism of log stacks such that the composites $S\rightrightarrows \cT\to T$ are equal, than the two morphisms are also equal.
\end{lemma}

We emphasize the ``equal'', and not merely isomorphic, even if $\cT$ is a stack. This is because the functor represented by $\cT$ on log schemes is actually valued in sets (see Remark \ref{remark:valued.in.sets}).

\begin{proof}
Let $f,g\colon S\to \cT$ be the two morphisms. By construction of $\cT$ (say it is the root stack $\sqrt[N]{T}$ corresponding to the Kummer morphism $\overline{M}_T\to N$) and the fact that the composites with $\cT\to T$ coincide, it follows that these are determined by a morphism of log schemes $ S\to T$, together with homomorphisms $f^\flat,g^\flat\colon N|_S \to \overline{M}_S$ commuting over $\Div_S$, and such that the composites $\overline{M}_T|_S\to  N|_S\rightrightarrows \overline{M}_S$ concide.

From this, since the first morphism is Kummer and $\overline{M}_S$ is torsion-free, we conclude that $f^\flat=g^\flat$, and therefore $f=g$.
\end{proof}

\begin{lemma}\label{lemma:log.unique.lift}
If $\cT\to T$ is a root stack and $t\to T$ is a log geometric point, then there is a canonical  isomorphism $t\stackrel{\cong}{\to} \cT\times_Tt$.
\end{lemma}

In other words, log geometric points of a root stack are the same as log geometric points of the underlying log scheme.

\begin{proof}
Write $t=\Spec(\Omega\to K)$. Note that, since the monoid $\Omega$ is uniquely divisible, the morphism $t\to T$ has a unique lift to $t\to \cT$: indeed, if $\overline{M}_{{T}}|_t\to A$ is a Kummer morphism, then there is a unique extension $A\to \Omega$ of the given $\overline{M}_{{T}}|_t\to \Omega$, and therefore a unique diagram $A\to \Omega\to \Div_{\Spec K}$ lifting $\overline{M}_{{T}}|_t\to \Omega\to \Div_{\Spec K}$.

It follows that the projection $\cT\times_T t\to t$ has a canonical section. The conclusion now follows from Lemma \ref{lemma:root.mono}, since $\cT\times_T t\to t$ will be a monomorphism with a section, hence an isomorphism.
\end{proof}

\begin{proposition}
Let $X\to Y$ be a morphism of fs log schemes, and $\cX\to \cY$ be an induced morphism between two root stacks, for some compatible systems of denominators on $X$ and $Y$.

Then for every log geometric point $y\to Y$ the natural  map $\cX\times_\cY y \to X\times_Y y$ between the log geometric fibers gives a homeomorphism of the underlying stacks.
\end{proposition}

\begin{proof}
First of all note that the diagram with Cartesian squares
$$
\xymatrix{
\cX\times_\cY y\ar[r]\ar[d] & \cX\ar[d]\\
y\ar[r]\ar[d]^\cong & \cY\ar[d]\\
y\ar[r] & Y,
}
$$
where the lower left map is an isomorphism thanks to the previous lemma, shows that the natural map $\cX\times_\cY y\to \cX\times_Y y$ is an isomorphism. 

We now check that the natural map $\cX\times_Y y\to X \times_Y y$ is a homeomorphism of the underlying stacks. Note first of all that, being a base-change of $\cX\to X$, it is a proper morphism (the appropriate diagram is Cartesian in the saturated sense, but the normalization map and closed embeddings are proper, so properness is preserved by passing to the saturated fiber product).

It suffices therefore to show that the morphism is injective and surjective on geometric points. Both follow from the fact that log geometric points lift uniquely along projections of root stacks.

For surjectivity, let $\Spec K\to X\times_Y y$ be a geometric point, that we can turn into a strict log geometric point $\Spec(\Omega \to K)\to X\times_Y y $. We have then that the composite $\Spec (\Omega \to K)\to X\times_Y y \to X$ gives a log geometric point of $X$. Since log geometric points lift along projections of root stacks, we have a lift $\Spec (\Omega \to K)\to \cX$. This in turn induces $\Spec(\Omega \to K)\to \cX\times_Y y$ lifting the given geometric point, thanks to the universal property of the fibered product.

As for injectivity, assume that we have two geometric points $x \to \cX\times_Y y$ and $x' \to \cX\times_Y y$ that map to the same point of $ X\times_Y y$. This means that there is a third geometric point $x''$ of $ X\times_Y y$ with a diagram $x\leftarrow x'' \to x'$ over $ X\times_Y y$, and we can turn $x,x',x''$ into log geometric points of $ X\times_Y y$ in the obvious way. We then obtain a commutative diagram of log schemes over $ X\times_Y y$, in particular also over $X$.

Using again the fact that log geometric points lift uniquely along the root stack projection $\cX\to X$, we can conclude that the diagram  $x\leftarrow x'' \to x'$ is also naturally over $\cX$. From the universal property of the fibered products, the log geometric point $x''$ lifts to a log geometric point of $\cX\times_Y y$, and the resulting diagram of log geometric points $x\leftarrow x'' \to x'$ of  $\cX\times_Y y$ shows that $x=x'$ as points of $|\cX\times_Y y|$.
\end{proof}

\begin{corollary}\label{cor:connected.root}
A morphism $X\to Y$ of fs log schemes is log geometrically connected if and only if the same is true for one (equivalently, any) induced morphism of compatible root stacks $\cX\to \cY$ of $X$ and $Y$.\qed
\end{corollary}

\subsection{Fk\'et morphisms via root stacks}\label{sec:correspondence}

We now explain how Kummer \'etale covers of a quasi-compact  fs log scheme $X$ can be identified with finite \'etale covers of some root stack of $X$.

The construction is outlined in \cite[Section 2]{Olsson}. In that paper, Olsson explains how to construct the finite \'etale cover, but doesn't prove that one can go back, and construct a Kummer \'etale cover of $X$ starting from a finite \'etale cover of some root stack $\sqrt[n]{X}$. In order to do that, we exploit one of the main results of \cite{irs}, stating that the Kummer \'etale topos of $X$ is equivalent to an appropriately defined \'etale site of the infinite root stack $\sqrt[\infty]{X}'$. 

We begin by recalling Olsson's construction. Assume that $Y\to X$ is a $\fket$ cover. In particular $Y$ is also quasi-compact, and we can pick $n\in \bN$ such that $\overline{M}_Y$ can be identified with a subsheaf of $\frac{1}{n}\overline{M}_X|_Y$, with $\chara k\nmid n$. One can then consider the root stack $\cY$ of $Y$ corresponding to the Kummer morphism $\overline{M}_Y\to \frac{1}{n}\overline{M}_X|_Y$. There is a strict morphism of root stacks $\cY\to \sqrt[n]{X}$, fitting into a commutative diagram of log stacks
$$
\xymatrix{
\cY\ar[r]\ar[d] & \sqrt[n]{X}\ar[d]\\
Y\ar[r] & X.
}
$$
Note that the upper map will induce, for any further root stack $\sqrt[m]{X}$ with $n\mid m$, a f\'et cover $\cY'\to \sqrt[m]{X}$ by pullback along the projection $\sqrt[m]{X}\to \sqrt[n]{X}$, and by passing to infinite root stacks, the finite \'etale morphism $\sqrt[\infty]{Y}'\to \sqrt[\infty]{X}'$ induced by $Y\to X$.

The last observation provides the link with the aforementioned equivalence of topoi proven in \cite{irs}. Let us briefly recall how that works as well: sending a k\'et morphism (not necessarily finite, e.g. this could be a strict open embedding) $Y\to X$ to the induced morphism $\sqrt[\infty]{Y}'\to \sqrt[\infty]{X}'$ gives a morphism of sites from the k\'et site of $X$ to the \emph{\'etale site} of $\sqrt[\infty]{X}'$, defined in the natural manner by considering representable \'etale maps. It is proven in \cite[Theorem 6.22]{irs} that the induced morphism of topoi is an equivalence (the proof, which is omitted in loc. cit. but is an obvious variant of the proof of \cite[Theorem 6.22]{irs}, works with prime-to-$\chara k$ root stacks, and in fact this restriction is actually needed to make the proof run through).

This also gives an equivalence between the categories of finite locally constant sheaves in the two topoi. These categories correspond (using standard descent arguments - see \cite[Proposition 3.13]{illusie-survey} and \cite[Proposition 4.3.6]{beyondfs} for the log side) to finite Kummer \'etale covers of $X$ and finite \'etale covers of $\sqrt[\infty]{X}'$.

By construction, the fk\'et cover $Y\to X$ will correspond to the finite \'etale cover $\sqrt[\infty]{Y}'\to \sqrt[\infty]{X}'$, which will descend to the $\cY\to \sqrt[n]{X}$ constructed by Olsson (note that there is an isomorphism $\fet(\sqrt[\infty]{X}')\cong \varinjlim_n \fet(\sqrt[n]{X})$, with $\chara k\nmid n$, constructed in the obvious way - as can be seen by standard arguments elaborating on \cite[Theorem 6.1]{irs}).

We record the outcome of this discussion in the following proposition.

\begin{proposition}\label{prop:correspondence}
There is an equivalence $$\fket(X)\cong \fet(\sqrt[\infty]{X}')\cong \varinjlim_{\chara k\nmid n} \fet(\sqrt[n]{X}),$$ that associates to a fk\'et cover $Y\to X$ the f\'et cover $\cY\to \sqrt[n]{X}$ between the root stacks determined by the Kummer extensions of $\overline{M}_Y$ and $\overline{M}_X$ into  $\frac{1}{n}\overline{M}_X$. \qed
\end{proposition}

\begin{notationnn}
In light of this correspondence, for the rest of the paper we will without further mention only consider root stacks for indices $n$ such that $\chara k\nmid n$.
\end{notationnn}

For future use, we record a compatibility of this construction with respect to pullbacks along saturated morphisms. 

\begin{proposition}\label{prop:pullback}
Let $X\to Y$ be a saturated morphism of fs log schemes with $Y$ quasi-compact, and $Z\to Y$ be a fk\'et morphism. Let $\cZ\to \sqrt[n]{Y}$ be the corresponding f\'et morphism for $n$ divisible enough, as in Proposition \ref{prop:correspondence}. Then the projection $\cZ\times_{\sqrt[n]{Y}}\sqrt[n]{X}\to \sqrt[n]{X}$ gives the f\'et morphism corresponding to the pullback $Z\times_Y X\to X$.

\end{proposition}

\begin{proof}
Let $n \in \bN$ be such that $\overline{M}_Z$ can be identified with a subsheaf of $\frac{1}{n}\overline{M}_Y|_Z$, as in the previous discussion. Recall (see Section \ref{sec:log.schemes}) that the log structure on the fibered product $Z'=Z\times_Y X$ has characteristic monoid $\overline{M}_{Z'}$ locally given by the coproduct of the diagram 
$$
\xymatrix{
\overline{M}_Y|_{Z'}\ar[r]\ar[d] & \overline{M}_Z|_{Z'} \\
\overline{M}_X|_{Z'} &
}
$$
made integral and saturated. These extra steps are not necessary here, because $X\to Y$ is a saturated morphism.

From this, starting from the obvious homomorphisms   $\overline{M}_Z|_{Z'}\to \frac{1}{n}\overline{M}_Y|_{Z'}\to \frac{1}{n}\overline{M}_X|_{Z'}$ and $\overline{M}_X|_{Z'}\to \frac{1}{n}\overline{M}_X|_{Z'}$, we obtain an induced morphism $\overline{M}_{Z'}\to \frac{1}{n}\overline{M}_X|_{Z'}$, which has the defining property of a Kummer morphism (thanks to the composite $\overline{M}_X|_{Z'}\to \overline{M}_{Z'}\to \frac{1}{n}\overline{M}_X|_{Z'}$ being the inclusion), except possibly for injectivity. After passing to stalks at a geometric point, this follows from the following lemma.

\begin{lemma}
Let $P, Q, R, S$ be sharp fs monoids with two morphisms $f\colon P\to Q$, $g\colon R\to S$, and two Kummer morphisms $P\hookrightarrow R$, $Q\hookrightarrow S$ (which we think of as inclusions), such that the two composites $P\to S$ coincide. Assume that $f$ is saturated. Then the induced morphism $Q\oplus_P R\to S$ is injective.
\end{lemma}

\begin{proof}
Recall that elements of the coproduct $Q\oplus_P R$ are formal sums $q+r$ with $q\in Q$ and $r\in R$, up to the equivalence relation generated by $(q+f(p))+r=q+(p+r)$ for $p\in P$.

Let us assume that $q+r$ and $q'+r'$ have the same image in $S$, i.e. $$q+g(r)=q'+g(r')$$ (recall that we are thinking of $Q$ as a submonoid of $S$ via the Kummer morphism). Now since $P\to R$ is Kummer, there exists a natural number $N$ and $p,p'\in P$ such that $Nr=p, Nr'=p'$. Moreover, note that $g(Nr)=f(p)$ and $g(Nr')=f(p')$. By multiplying the relation $q+g(r)=q'+g(r')$ in $S$ by $N$, we obtain $$Nq+f(p)=Nq'+f(p'),$$ and this equality holds in $Q$ (since $Q\hookrightarrow S$ is injective).

Since $f$ is integral, there exist $p_1,p_1' \in P$ and $q''\in Q$ such that $Nq=f(p_1)+q'', Nq'=f(p_1')+q''$ and $p+p_1=p'+p_1'$. Now we can see that $N(q+r)=N(q'+r')$ in $Q\oplus_P R$, since
$$
Nq+Nr=(q''+f(p_1))+p=q''+(p_1+p)=
$$
$$=q''+(p_1'+p')=(q''+f(p_1'))+p'=Nq'+Nr'.
$$
To conclude it is enough to note that $Q\oplus_P R$ is torsion-free, since it is fs (recall that $f$ is saturated) and sharp (sharpness follows from \cite[Proposition 4.2.5-(3.)]{ogus}).
\end{proof}

We now show that $\cZ\times_{\sqrt[n]{Y}}\sqrt[n]{X}$ is the root stack $\cZ'$ of $Z'$ with respect to the system of denominators $\overline{M}_{Z'}\to \frac{1}{n}\overline{M}_X|_{Z'}$, concluding the proof.

We clearly have compatible morphisms $\cZ'\to \cZ$, $\cZ'\to \sqrt[n]{X}$, and $\cZ'\to \sqrt[n]{Y}$, described at the level of functors by restricting the DF structure $\frac{1}{n}\overline{M}_X|_T\to \Div_T$ on a $\cZ'$-scheme $T$ to the appropriate subsheaf. These induce a morphism $\cZ'\to \cZ\times_{\sqrt[n]{Y}}\sqrt[n]{X}$.

This is in fact an equivalence: given a scheme $T$ with a morphism $T\to \cZ\times_{\sqrt[n]{Y}}\sqrt[n]{X}$, we obtain in particular compatible morphisms $T\to \underline{Z}$ and $T\to \underline{X}$, together with compatible DF structures $\frac{1}{n}\overline{M}_X|_T\to \Div_T$ and $\frac{1}{n}\overline{M}_Y|_T\to \Div_T$, extending the pullbacks of the DF structures of $X,Y$ and $Z'$.

From the morphisms of schemes we get a map $T\to \underline{Z'}$ (recall that this is the usual fibered product of schemes, since $X\to Y$ is saturated), and the DF structure $\frac{1}{n}\overline{M}_X|_T\to \Div_T$ gives an object of the root stack $\cZ'$. We leave the remaining routine verifications to the reader.
\end{proof}

\subsection{Full faithfulness of pullback}

We now prove a logarithmic analogue of full faithfulness of pullback on \'etale morphisms for certain morphisms of schemes \cite[IX, Corollaire 3.4]{SGA1}, which will be used later in a couple of instances. Let us first recall the precise statement.

\begin{proposition}(\cite[IX, Corollaire 3.4]{SGA1})
Let $f\colon X\to Y$ be a geometrically connected and universally submersive morphism of schemes. Then if $Y$ is connected, so is $X$, and the pullback functor $f^*\colon \et(Y)\to \et(X)$ is fully faithful.
\end{proposition}

Here is the logarithmic version that we will need. Note that we do not concern ourselves with the full \'etale site, and only look at finite Kummer \'etale covers instead.

\begin{proposition}\label{prop:ff}
Let $f\colon X\to Y$ be a universally submersive, saturated and log geometrically connected morphism of fs log schemes. Then the functor $f^*\colon \ket(Y)\to \ket(X)$ is fully faithful.

In particular, if $Y'\to Y$ is a Galois k\'et cover with group $G$, then the pullback $f^*Y'\to X$ is also a Galois k\'et cover with group $G$.
\end{proposition}

We will apply this to the case of a log smooth, saturated and locally finitely presented morphism $f$ (which will be universally submersive, since saturated and log smooth implies flat, and hence $f$ will be fppf).

For the proof, we will use the correspondence outlined in the previous subsection to reduce to a question about morphisms of root stacks, and then we will apply the following Lemma.

\begin{lemma}\label{lemma:ffstacks}
Let $f\colon \cX\to \cY$ be a geometrically connected and universally submersive morphism of Deligne--Mumford stacks with finite relative inertia. Then the pullback functor $f^*\colon \fet(\cY)\to \fet(\cX)$ is fully faithful.
\end{lemma}

\begin{proof}
Consider the relative coarse moduli space $\cX\to \cX'\to \cY$ \cite[Section 3]{AOV}. The pullback along $f$ is then the composition of the two pullbacks. Since $\cX'\to \cY$ is representable and universally submersive (because the composite $\cX\to \cX'\to \cY$ is), a standard descent argument shows that fully faithfulness of pullback follows immediately from the case of schemes.

As for the morphism $\cX\to \cX'$, again by descent we can assume $\cX'=X$ is a connected scheme, so that we are looking at a coarse moduli space map $\cX\to X$. Moreover, full faithfulness of the pullback functor on f\'et covers is equivalent to the fact that the induced homomorphism $\pi_1^\et(\cX, x')\to \pi_1^\et(X,x)$ of \'etale fundamental groups is surjective, where $x'$ is some geometric point of $\cX$ and $x$ is its image in $X$ (the proof of this is completely formal, see for example \cite[\href{https://stacks.math.columbia.edu/tag/0BN6}{Tag 0BN6}]{stacks-project}).

On the other hand, in the case of a coarse moduli space morphism it is known that the homomorphism of fundamental groups is surjective \cite[Theorem 7.11]{noohi}.
\end{proof}

\begin{proof}[Proof of Proposition \ref{prop:ff}]

Let us fix two objects $Z,Z'\in \fket(Y)$. By Proposition \ref{prop:correspondence}, we have corresponding objects $\cZ,\cZ'\in \fet(\sqrt[n]{Y})$ for some $n$ sufficiently divisible. Moreover, because of the equivalence of categories between fk\'et and f\'et covers, we have $\Hom_Y(Z,Z')=\Hom_{\sqrt[n]{Y}}(\cZ,\cZ')$, where the first Hom is in the logarithmic category, the second is just morphisms of stacks.

Now by Proposition \ref{prop:pullback}, the morphisms of root stacks of $Z\times_Y X$ and $Z'\times_Y X$ corresponding to the fk\'et covers $Z\times_Y X\to X$ and $Z'\times_Y X\to X$ are precisely the projections $\cZ\times_{\sqrt[n]{Y}}\sqrt[n]{X}\to \sqrt[n]{X}$ and $\cZ'\times_{\sqrt[n]{Y}}\sqrt[n]{X}\to \sqrt[n]{X}$. Using Lemma \ref{lemma:ffstacks} we deduce that
$$
\Hom_{\sqrt[n]{Y}}(\cZ,\cZ')=\Hom_{\sqrt[n]{X}}(\cZ\times_{\sqrt[n]{Y}}\sqrt[n]{X},\cZ'\times_{\sqrt[n]{Y}}\sqrt[n]{X})
$$
and finally, by the same reasoning as in the previous paragraph, we also  have that $\Hom_{\sqrt[n]{X}}(\cZ\times_{\sqrt[n]{Y}}\sqrt[n]{X},\cZ'\times_{\sqrt[n]{Y}}\sqrt[n]{X})=\Hom_X(Z\times_Y X, Z'\times_Y X)$. 

It follows that the composite $\Hom_Y(Z,Z')\to \Hom_X(f^*Z, f^*Z')$ is a bijection.
\end{proof}

\section{Log Stein factorization}

In this section we prove the following special case of Stein factorization for morphisms of fs log schemes.

\begin{proposition}\label{prop:stein.factorization}
Let $f\colon X\to S$ be a proper, finitely presented, log smooth, saturated, log geometrically connected morphism of fs log schemes over $k$, and let $g\colon Y\to X$ be a fk\'et morphism. Then the composite $f\circ g\colon Y\to S$ factors uniquely as
\[
\xymatrix{
Y\ar[r]^{f'} & T\ar[r]^{g'} & S
}
\]
where $f'$ is proper, log geometrically connected, and $g'$ is fk\'et.\end{proposition}

Such a factorization will be called a \emph{log Stein factorization}. It is functorial with respect to morphisms in $\fket(X)$, and it is left adjoint to the pullback functor $f^*$, as we show later (Proposition \ref{prop:stein.properties}). 

\begin{remark}
We stress that our assumptions are quite more restrictive than the standard ones for the existence of the Stein factorization in the non-logarithmic context, and indeed the previous proposition should be a special case of a more general result requiring weaker assumptions. Our method of proof isn't very likely to generalize, since it relies heavily on the fact that we are looking at fk\'et morphisms, and can pass to root stacks to ``hide'' the log structures and employ non-log methods.
\end{remark}

\begin{proof}
In order to construct the factorization we will pass to root stacks and use the ``standard'' Stein factorization for morphisms of stacks. This is worked out for example in \cite[Theorem 11.3]{olsson.sheaves} in the locally Noetherian case, but for convenience we recall the construction next.

\subsection*{Stein factorization for algebraic stacks} Assume that $f\colon \cX\to \cY$ is a proper, finitely presented morphism of algebraic stacks with finite relative inertia, with $\cY$ quasi-compact. We also assume that $f$ is tame, i.e. the relative inertia has linearly reductive geometric fibers. This ensures that the relative coarse moduli space \cite[Section 3]{AOV} is compatible with base-change. 

The pushforward $f_*\cO_\cX$ is a quasi-coherent sheaf of algebras on $\cY$, and we can consider the relative spectrum $\cZ=\underline{\Spec}_\cY f_*\cO_\cX$. We then have a factorization of $f$ as $\cX\to \cZ\to \cY$, where the first map is proper with geometrically connected fibers, and the second one is an integral morphism, as we now show. Note that, thanks to standard limit arguments (see \cite{rydh}), we can also assume that $\cY$ is Noetherian. Consider the factorization $\cX\stackrel{g}{\to} \cX'\stackrel{h}{\to} \cY$ through the relative coarse moduli space (so that $\cX'\to \cY$ is the initial representable morphism factoring $f$), so that we have a factorization $\cX\to\cX'\to \cZ$ of $\cX\to\cZ$. Note moreover that, since $g_*\cO_{\cX}\cong \cO_{\cX'}$, we also have $\cZ\cong \underline{\Spec}_\cY h_*\cO_{\cX'}$, and that $h\colon \cX'\to \cY$ is proper \cite[\href{https://stacks.math.columbia.edu/tag/0DUZ}{Tag 0DUZ}]{stacks-project}. Since $h$ is representable, it follows by the case of schemes and by descent that $\cZ\to \cY$ is integral, and that $\cX'\to \cZ$ is proper with geometrically connected fibers. Since $\cX\to \cX'$ is also proper with geometrically connected fibers, the claim follows. 

If $f\colon \cX\to \cY$ is also flat with geometrically reduced fibers, then $\cZ\to \cY$ is actually finite \'etale (so it is a f\'et cover). This follows from the fact that if $f$ is flat with geometrically reduced fibers, then $\cX'\to \cY$ also is \cite[Theorem 4.16, (viii) and (ix)]{alper}, and this implies that $\cZ\to \cY$ is finite \'etale by descent from the case of schemes \cite[\href{https://stacks.math.columbia.edu/tag/0BUN}{Tag 0BUN}]{stacks-project}.

\begin{remark}\label{rmk:cc.of.fibers}
Later on, we are going to make use of the fact that, at least when $f$ is proper, finitely presented, flat and with reduced geometric fibers, the geometric points of $\cZ$ are in bijection with the connected components of the geometric fibers of $f$. For a proof see \cite[Proposition 3.2.5]{romagny}.
\end{remark}

\subsection*{Construction of the Stein factorization} In view of the claimed uniqueness and strict \'etale descent, we can localize on $S$ and assume that it is quasi-compact (for example affine). As in Section \ref{sec:correspondence}, choose $n\in \mathbb{N}$ such that the map $\overline{M}_Y$ can be identified with a submonoid of $\frac{1}{n}\overline{M}_X|_Y$, and consider the induced strict f\'et morphism $\cY\to \sqrt[n]{X}$.

For the same $n$, consider the root stack $\sqrt[n]{S}\to S$, and  the morphism $h\colon \cY\to \sqrt[n]{S}$ obtained by composing the f\'et morphism $\cY\to \sqrt[n]{X}$ with $\sqrt[n]{X}\to \sqrt[n]{S}$. Note that $h$ is flat, since $\cY\to \sqrt[n]{X}$ is \'etale and $\sqrt[n]{X}\to \sqrt[n]{S}$ is flat, because it is log smooth (this follows for example by noting that the composite $\sqrt[n]{X}\to X\to S$ is log smooth, and $\sqrt[n]{S}\to S$ is log \'etale) and integral (because $X\to S$ is). Moreover, the map $h$ is also proper, since $\cY\to \sqrt[n]{X}$ is finite, hence proper, and $ \sqrt[n]{X}\to\sqrt[n]{S}$ is proper, which follows again from the fact that the composite $\sqrt[n]{X}\to X\to S$ and the map $\sqrt[n]{S}\to S$ are both proper. Finally, since $X\to S$ is saturated, the same is true for $\sqrt[n]{X}\to\sqrt[n]{S}$, and therefore $h$ has geometrically reduced fibers.

Consider now the Stein factorization $\cY\stackrel{\alpha}{\to} \cT\stackrel{\beta}{\to}\sqrt[n]{S}$ of the morphism $h$ described above, so that the morphism $\alpha$ is proper with geometrically connected fibers, and $\beta$ is f\'et. Equip $\cT$ with the pullback log structure from $\sqrt[n]{S}$, so that the factorization is also a factorization in maps of log stacks.

Using again the correspondence outlined in Section \ref{sec:correspondence}, let $T$ be the coarse moduli space of the stack $\cT$, with the maps $f'\colon Y\to T$ and $g'\colon T\to S$ induced by $\alpha$ and $\beta$ respectively. Equipping $T$ with the pushforward log structure from $\cT$, and using the fact that the log structures of $Y$ and $X$ are also the pushforward of the ones on $\cY$ and $\sqrt[n]{X}$ respectively (this is proven in \cite[Lemma 2.11]{Olsson}, which also ensures that the resulting log structure on $T$ is fine and saturated), these morphisms become morphisms of log schemes. We now note that $g'$ is fk\'et (because of Proposition \ref{prop:correspondence}), $f'$ is proper (this is clear, since $g'\circ f'$ is proper, and $g'$ is separated) and log geometrically connected, thanks to Corollary \ref{cor:connected.root}. This gives the required factorization.

Now we can also prove the claimed uniqueness. It is enough to prove it locally on $S$ so we can assume, as we did before, that $S$ is quasi-compact. 
Assume now that $Y\to T\to S$ is a factorization, with $Y\to T$ proper and with log geometrically connected fibers, and $T\to S$ a fk\'et morphism.

By Proposition \ref{prop:correspondence}, the morphism $T\to S$ is determined by a finite \'etale morphism  $\cT\to \sqrt[n]{S}$ for some $n$. Moreover, we can assume that for the same $n$ we also have a diagram $\cY\to \sqrt[n]{X}\to \sqrt[n]{S}$ lifting $Y\to X\to S$. Now the map $\cY\to \sqrt[n]{S}$ factors as $\cY\stackrel{f}{\to} \cT\stackrel{g}{\to} \sqrt[n]{S}$, as is evident from the functorial description of $\cT$ as a root stack. It remains to show that $\cT$ can be canonically identified with the relative spectrum of $(g\circ f)_*\cO_\cY$.

Note first of all that the morphism $f\colon \cY\to \cT$ of log stacks is saturated, because the morphism $f^\flat\colon f^*M_\cT\to M_\cY$ of log structures is induced by the morphism $\sqrt[n]{X}\to \sqrt[n]{S}$, which is saturated because $X\to S$ is, by assymption (recall that both $\cY\to \sqrt[n]{X}$ and $\cT\to \sqrt[n]{S}$ are strict). Therefore, the underlying schemes of the log geometric fibers of $f$ are isomorphic to the geometric fibers. Putting this together with Corollary \ref{cor:connected.root}, It follows that $f$ is geometrically connected.

 The above claim about $\cT$ now follows  from the fact that $\cT\to \sqrt[n]{S}$ is affine, so $\cT\cong \underline{\Spec}_{\sqrt[n]{S}} g_*\cO_\cT$, and that $f_*\cO_\cY\cong \cO_\cT$, because $f$ is flat, proper, of finite presentation and with geometrically connected and geometrically reduced fibers (because the composite $\cY \to \sqrt[n]{S}$ has geometrically reduced fibers and $g$ is \'etale), see \cite[\href{https://stacks.math.columbia.edu/tag/0E0L}{Tag 0E0L}]{stacks-project}.
 \end{proof}

\begin{proposition}\label{prop:stein.properties}
The construction described above defines a Stein factorization functor $\St\colon \fket(X)\to \fket(S)$, which is left adjoint to the pullback $f^*\colon \fket(S)\to \fket(X)$. Moreover, the counit of the adjunction is an isomorphism.
\end{proposition}

For simplicity of notation, If $Y\to X$ is a fk\'et cover we will just write $\St(Y)$ instead of the more precise $\St(Y\to X)$.

\begin{proof}
Let $Y\to Y'$ be a morphism of fk\'et covers of $X$. There is an induced morphism $\St(Y)\to \St(Y')$, simply coming from the fact that the map $\cY\to \cY'$ over $\sqrt[n]{S}$ (with $n$ divisible enough, as usual) induces a morphism $\cT\to \cT'$ between the relative affinizations, which then passes to the coarse moduli spaces, compatibly with the log structures. Compatibility with composition is also checked similarly. 

We now show adjointness. Note that a morphism $Y\to f^*U=U\times_S X$ over $X$ is the same as a morphism $Y\to U$ over $S$. Both $Y\to S$ and $U\to S$ lift to maps $\cY\to \sqrt[n]{S}$ (coming from the morphism $\cY\to \sqrt[n]{X}$ of Proposition \ref{prop:correspondence}, composed with $\sqrt[n]{X}\to \sqrt[n]{S}$) and $\cU\to \sqrt[n]{S}$ (again from Proposition  \ref{prop:correspondence}), and by the functorial description of the root stacks, we also have a lifting $\cY\to \cU$ over $\sqrt[n]{S}$. Since $\cU\to \sqrt[n]{S}$ is affine, we obtain an induced morphism $\cT\to \cU$ from the relative affinization $\cT$ of $\cY\to \sqrt[n]{X}$ (which is f\'et over $\sqrt[n]{S}$). Taking coarse moduli spaces and induced log structures, this induces a morphism $\St(Y)\to U$. The inverse function is immediate to construct, and it is clear that the two are inverses.

Finally, the counit $\St(f^*U)\to U$ is an isomorphism because the log Stein factorization of $f^*U\to S$ is clearly the original $U\to S$, since $f^*U\to U$ is proper and log geometrically connected.
\end{proof}

\section{The homotopy exact sequence}

We are now ready to prove the exactness of the homotopy sequence for the logarithmic fundamental group.

\begin{theorem}\label{theorem:main}
Let $f\colon X\to S$ be a proper, finitely presented, log smooth, saturated, log geometrically connected morphism of connected fs log schemes over $k$, and let ${x}\to X$ be a log geometric point with image ${s}\to S$. Then there is an exact sequence
$$
\xymatrix{
\pi_1^\ket(X_{{s}},{x})\ar[r] & \pi_1^\ket(X,{x})\ar[r] & \pi_1^\ket(S,{s})\ar[r] & 1.
}
$$
\end{theorem}

\begin{proof}
The morphisms in the sequence are of course the ones induced by the maps $X_{{s}}\to X$ and $X\to S$. Note also that the log geometric point $x$ of $X$ will lift uniquely to the log geometric fiber $X_s$. The proof of the theorem will follow along the lines of the one for the statement for schemes in \cite[\href{https://stacks.math.columbia.edu/tag/0BUM}{Tag 0BUM}]{stacks-project}.

Exactness on the right,  i.e. surjectivity of $\pi_1^\ket(X,{x})\to \pi_1^\ket(S,{s})$, follows directly from the fact that $f^*$ is fully faithful, thanks to Proposition \ref{prop:ff} (this alternatively also follows from the fact that $f^*$ is right adjoint to the Stein factorization functor and the counit of the adjunction is an isomorphism - see Proposition \ref{prop:stein.properties}), and  \cite[\href{https://stacks.math.columbia.edu/tag/0BN6}{Tag 0BN6}]{stacks-project}.

The fact that the composite $\pi_1^\ket(X_{{s}},{x})\to \pi_1^\ket(X,{x})\to \pi_1^\ket(S,{s})$ is trivial follows from the obvious fact that if $U\to S$ is a fk\'et cover, then the image in $\fket(X_{{s}})$ is a trivial covering (since $X_{{s}}\to S$ factors through ${s}\to S$, and $U$ becomes trivial on ${s}$, which is clear because $U$ will split over some k\'et cover of $S$, and $s$ will lift to said cover), and from \cite[\href{https://stacks.math.columbia.edu/tag/0BS8}{Tag 0BS8}]{stacks-project}.

Let us turn to exactness in the middle. Thanks to \cite[\href{https://stacks.math.columbia.edu/tag/0BS9}{Tag 0BS9}]{stacks-project} and \cite[\href{https://stacks.math.columbia.edu/tag/0BTS}{Tag 0BTS}]{stacks-project}, it suffices to prove that if a connected object $Y\to X$ of $\fket(X)$ is such that the pullback cover $Y\times_X X_{{s}}\to X_{{s}}$ has a section, then it is a trivial cover.

Consider the Stein factorization $\St(Y)$ of $Y\to X\to S$, and the induced morphism $Y\to f^*\St(Y)=\St(Y)\times_S X$, which is fk\'et, being a morphism between fk\'et log schemes over $X$. We will show that this is of degree 1, hence an isomorphism. It will follow then that $Y\times_X X_{{s}}\to X_{{s}}$ is the trivial cover. Indeed, we have $Y\times_X X_{{s}}\cong Y\times_S {s}=Y_{{s}}$ and $Y_s\to X_s$ can be identified with the projection $Y_s\cong \St(Y)_s\times_s X_s\to X_s$ via the cartesian square
$$\xymatrix{
Y_s\ar[r]\ar[d] & X_s\ar[d]\\
\St(Y)_s\ar[r] & s
}
$$
obtained by the one witnessing the isomorphism $Y\cong \St(Y)\times_S X$ by pulling back along $s\to S$. Finally note that $\St(Y)_s\to s$ is a trivial cover, because $s$ is a log geometric point.

Denote $\St(Y)$ by $T$. In order to show that the degree of $Y\to T\times_S X$ is 1, we can localize on $S$ and assume that it is quasi-compact, and therefore the Stein factorization is determined by the maps $\cY\to \cT\to \sqrt[n]{S}$ of root stacks.

Note that in the diagram
$$
\xymatrix{
\cY\ar[r]\ar[d] & \cT\times_{\sqrt[n]{S}} \sqrt[n]{X}\ar[d]\\
Y\ar[r] & T\times_S X
}
$$
the rightmost vertical map is a root stack thanks to Proposition \ref{prop:pullback}, and the upper horizontal map is f\'et (since both $\cY\to \sqrt[n]{X}$ and $\cT\times_{\sqrt[n]{S}} \sqrt[n]{X}\to \sqrt[n]{X}$ are f\'et). It suffices therefore to check that the f\'et morphism $\cY\to \cT\times_{\sqrt[n]{S}} \sqrt[n]{X}$ has degree 1.

This can be done as follows: first note that the geometric points of $\cT$ correspond to connected components of the geometric fibers of $\sqrt[n]{X}\to \sqrt[n]{S}$ (see Remark \ref{rmk:cc.of.fibers}). Since this morphism is saturated, the geometric fibers are the underlying stacks to the log geometric fibers, and if ${s}\to S$ is a log geometric point, lifting uniquely to ${s}\to \sqrt[n]{S}$ thanks to Lemma \ref{lemma:log.unique.lift}, then, as we show next, we have $(\sqrt[n]{X})_{{s}}\cong \sqrt[n]{X_{{s}}}$.

\begin{lemma}
With the assumptions of \ref{theorem:main}, there is a canonical isomorphism $(\sqrt[n]{X})_{{s}}\cong \sqrt[n]{X_{{s}}}$.
\end{lemma}

\begin{proof}
There is a natural map $\sqrt[n]{X_s}\to (\sqrt[n]{X})_s$, obtained from the universal property of the fibered product $(\sqrt[n]{X})_s=\sqrt[n]{X}\times_S s$, starting from $\sqrt[n]{X_s}\to \sqrt[n]{X}$ and $\sqrt[n]{X_s}\to X_s\to s$. It is easy to see that this is an isomorphism at the functorial level: it sends a morphism $T\to \sqrt[n]{X_s}$, corresponding to the morphism of schemes $T\to \underline{X_s}$ together with a DF structure $\frac{1}{n}\overline{M}_{X_s}|_T\to \Div_T$ lifting the pullback of the DF structure $ \overline{M}_{X_s}|_T\to \Div_T$ of $X_s$, to the morphism $T\to \sqrt[n]{X}$, corresponding to the induced morphism of schemes $ T\to \underline{X_s}\to \underline{X}$ together with the DF structure $\frac{1}{n}\overline{M}_X|_T\to \frac{1}{n}\overline{M}_{X_s}|_T\to \Div_T$. Since $T\to S$ factors through $s\to S$, this map will factor through $(\sqrt[n]{X})_s\to \sqrt[n]{X}$, giving the desired $T\to (\sqrt[n]{X})_s$.

To go in the other direction, start with a morphism $T\to (\sqrt[n]{X})_s$. The induced compatible morphisms of schemes $T\to \underline{X}$ and $T\to \underline{s}$ will uniquely determine a morphism $T\to \underline{X_s}$, and the fact that a lift $\frac{1}{n}\overline{M}_X|_T\to \Div_T$ of the pullback of the DF structure of $X$ will extend uniquely to a lift $\frac{1}{n}\overline{M}_{X_s}|_T \to \Div_T$ follows from noting that by construction (and saturatedness of $X\to S$) the sheaf $\frac{1}{n}\overline{M}_{X_s}|_T$ is the coproduct in the diagram
$$
\xymatrix{
\frac{1}{n}\overline{M}_S|_T\ar[r]\ar[d] & \frac{1}{n}\overline{M}_X|_T\ar[d] \\
\frac{1}{n}\overline{M}_s|_T\ar[r] & \frac{1}{n}\overline{M}_{X_s}|_T
}
$$
and, since $\overline{M}_s$ is uniquely divisible (recall that $s$ is a log geometric point), we have an inverse $\frac{1}{n}\overline{M}_s|_T\cong \overline{M}_s|_T$ to the inclusion $ \overline{M}_s|_T\to  \frac{1}{n}\overline{M}_s|_T$. Composing this with the given $ \overline{M}_s|_T\to \overline{M}_{X_s}|_T\stackrel{L}{\to} \Div_T$, from the universal property of the coproduct we obtain the desired DF structure $\frac{1}{n}\overline{M}_{X_s}|_T\to \Div_T$. This produces a $T$-point of $\sqrt[n]{X_s}$. Routine verifications show that these constructions are inverses.
\end{proof}

It follows now that the given section $X_{{s}}\to Y\times_X X_{{s}}$ induces a section $(\sqrt[n]{X})_{{s}}\to \cY\times_{\sqrt[n]{X}} (\sqrt[n]{X})_{{s}}$, because the diagram
$$
\xymatrix{
\cY\times_{\sqrt[n]{X}} (\sqrt[n]{X})_{{s}}\ar[r]\ar[d] &  (\sqrt[n]{X})_{{s}}\ar[d]\\
Y\times_X X_{{s}}\ar[r] & X_{{s}}
}
$$
is Cartesian: indeed, if $T$ is a log scheme equipped with morphisms $T\to Y\times_X X_{{s}}$ and $T\to (\sqrt[n]{X})_{{s}}$ such that the composites $T\to X_{{s}}$ agree, then in particular we obtain a morphism $T\to \sqrt[n]{X}$ (by composing with $(\sqrt[n]{X})_{{s}}\to \sqrt[n]{X}$) and a morphism $T\to Y$, and since $\cY\to \sqrt[n]{X}$ is strict, we also obtain an induced morphism $T\to \cY$ from the functorial description of $\cY$ as a root stack of $Y$. These data assemble to (uniquely) give an induced morphism $T\to \cY\times_{\sqrt[n]{X}} (\sqrt[n]{X})_{{s}}$.

Call $\cZ$ the connected component of $\cY\times_{\sqrt[n]{X}} (\sqrt[n]{X})_{{s}}\cong \cY\times_{\sqrt[n]{S}}{s}$ given by the image of this section, and ${z}$ the corresponding geometric point of $\cT$. Now if ${x}$ is any geometric point of $X_{{s}}$, the preimage in $\cY$ of the point of $\cT\times_{\sqrt[n]{S}} \sqrt[n]{X}$ corresponding to the pair $({z},{x})$ consists of exactly one point. Since $\cY\to \cT\times_{\sqrt[n]{S}} \sqrt[n]{X}$ is a finite \'etale morphism we can conclude that it has degree 1 (note that $T\times_S X$ is connected, and therefore $\cT\times_{\sqrt[n]{S}} \sqrt[n]{X}$ is connected as well, thanks to the surjectivity of $\pi_1^\ket(X,{x})\to \pi_1^\ket(S,{s})$ that we already proved \cite[\href{https://stacks.math.columbia.edu/tag/0BN6}{Tag 0BN6}]{stacks-project}). This concludes the proof of exactness in the middle, and therefore also the proof of the theorem.
\end{proof}

\begin{remark}
It is possible that one could give an alternative proof of Theorem \ref{theorem:main} by working directly with the infinite root stacks, since the log fundamental group is naturally identified with the \'etale fundamental group of the infinite root stack \cite[Theorem 6.4]{KNIRS}. Going down this path, one encounters some foundational issues, since infinite root stacks are only ``very weakly'' algebraic (they only have fpqc atlases), so we opted to use the inverse system and the root stacks at finite level.

\end{remark}

We now show that when $S$ is a log point, the sequence is actually short exact (and less assumptions are required on the morphism $X\to S$). This of course is morally due to the fact that log points are $K(\pi,1)$ spaces (over $\bC$, their ``analytification'' - i.e. the Kato--Nakayama space \cite{KN} - is a real torus of finite dimension).

\begin{theorem}\label{thm:left.exact}
Let $f\colon X\to S$ be a quasi-compact, log geometrically connected morphism of connected fs log schemes over $k$, where $S$ is a log point, and let ${x}\to X$ be a log geometric point with image ${s}\to S$. Then there is an exact sequence
$$
\xymatrix{
1\ar[r] &\pi_1^\ket(X_{{s}},{x})\ar[r] & \pi_1^\ket(X,{x})\ar[r] & \pi_1^\ket(S,{s})\ar[r] & 1.
}
$$
\end{theorem}

\begin{proof}
We will adapt the proof of \cite[IX 6]{SGA1} for the ``fundamental short exact sequence'', that relates the \'etale fundamental groups of a scheme over a field with the one of the base-change to a separable closure and the absolute Galois group of the base field.

Write $S=\Spec(P\to k)$, where $P$ is a sharp fs monoid, and let ${s}$ be the strict log geometric point ${s}\to S$, where ${s}=\Spec(P_\bQ \to k)$. This might be different than the $s$ in the statement, but by independence of the fundamental group from the base point (see \cite[Theorem 1.7 and Remark 1.8]{Olsson}), we reduce to proving that the sequence
$$
\xymatrix{
1\ar[r] & \pi_1^\ket(X_{{s}},{x})\ar[r] & \pi_1^\ket(X,{x})\ar[r] & \pi_1^\ket(S,{s})\ar[r] & 1
}
$$
for this strict log geometric point $s\to S$ is exact.

Let $L/k$ be a finite Galois extension of $k$, and $n\in \mathbb{N}$. Denote by $S^L_n$ be the k\'et cover $$S_n^L=\Spec\left(\frac{1}{n}P\to L\left[\frac{1}{n}P\right]/(P)\right)\to S,$$ and by $\stackrel{\circ}{S_n^L}$ its reduction with the induced log structure (which is a log point, since all elements of $\frac{1}{n}P$ are nilpotent). Note that ${s}\cong \varprojlim_{n,L} \stackrel{\circ}{S^L_n}$ as log schemes, and therefore $X_{{s}}\cong \varprojlim_{n,L} \stackrel{\circ}{X_n^L}$, where we set $X_n^L=X\times_S S_n^L$ and $\stackrel{\circ}{X_n^L}=X\times_S \stackrel{\circ}{S_n^L}$. The inverse system is indexed by pairs $(n,L)$ of a  natural number and a finite Galois extension of $k$, and there is an arrow $(n,L)\to (m,K)$ exactly when $n\mid m$ and  $L\subseteq K$. It follows, as in \cite[IX 6]{SGA1}, that
$$
\pi_1^\ket(X_{{s}},{x})\cong \varprojlim_{n,L} \pi_1^\ket(\stackrel{\circ}{X_n^L}, x_n^L)
$$
via the natural map, and where $x_n^L$ is the image of ${x}$ in $\stackrel{\circ}{X_n^L}$.

Now by topological invariance of the k\'et site \cite[Theorem 2.8]{illusie-survey} (see also \cite[Proposition 4.1.7]{beyondfs}), we have isomorphisms $$ \pi_1^\ket(\stackrel{\circ}{X_n^L}, x_n^L)\cong \pi_1^\ket(X_n^L, x_n^L),$$ which are clearly compatible with the various projection maps $X_m^L\to X_n^L$ for $n\mid m$.

Now note that $\mu_n(P)$, the Cartier dual of the cokernel of the map $P^\gp\to \frac{1}{n}P^\gp$, and the Galois group $G_L$ of $L/k$ act on the Galois cover $S_n^L$, and the quotient is $S$. The actions clearly commute, so it follows that the automorphism group of $S_n$ over $S$ is $\mu_n(P)(k)\times G_L$, and $\pi_1^\ket(S,s)\cong \varprojlim_{n,L} \mu_n(P)(k)\times G_L $. Now the pullback $X_n^L$ is also a Galois cover of $X$, and thanks to Proposition \ref{prop:ff}, we have that its automorphism group is again $\mu_n(P)(k)\times G_L$.

From \cite[V 6.13]{SGA1} we obtain then an exact sequence
$$
\xymatrix{
1\ar[r] & \pi_1^\ket(X_n^L,x_n^L)\ar[r] & \pi_1^\ket(X,{x})\ar[r] & \mu_n(P)(k)\times G_L \ar[r] & 1
}
$$
for each $n\in \bN$ and finite Galois extension $L/k$. These exact sequences are compatible in the obvious sense, and by taking the inverse limit over $(n,L)$  we obtain the short exact sequence in the statement.
\end{proof}

\begin{remark}

If $S$ is a log point, the group $\pi^\ket_1(X_s,x)\cong \ker (\pi^\ket_1(X,x)\to \pi_1^\ket(S,s))$ is in some regards the correct fundamental group of the log scheme ``$X$ relative to $S$''. The point is that full log fundamental group $\pi_1^\ket(X,x)$ also sees nontrivial loops in $S$, and is therefore too large for some purposes, for example when one has to deal with the fundamental group of fibers of a family.

If $k=\bC$, this can be visualized (at least with the assumptions of Theorem \ref{theorem:main}) by looking at the induced morphism of Kato--Nakayama spaces \cite{KN}. In the case of a morphism $X\to S$ as in Theorem \ref{theorem:main}, the space $S_\log$ is a real torus of dimension equal to the rank of $\overline{M}_S^\gp$, and $X_\log\to S_\log$ is a fiber bundle (thanks to \cite[Theorem 0.3]{nakayama-ogus}). Considering fundamental groups, there are going to be loops in $X_\log$ that map to non-trivial loops of $S_\log$. Passing to the kernel of the map of fundamental groups corresponds to restricting to a fiber, thanks to the exact sequence of a fibration in topology and the fact that real tori are $K(\pi,1)$ spaces.

Moreover, in line with comparison theorems between the log \'etale fundamental group and the topological fundamental group of the Kato--Nakayama space (see for example \cite[Theorem 6.4]{KNIRS}), we can show that the profinite completion of the kernel of $\pi_1(X_\log,x')\to \pi_1(S_\log, s')$ is isomorphic to the group $\pi_1^\ket(X_s,x)$ for log geometric points $s\to S$ and $x\to X_s$ and point $s'\in S_\log$ and $x'\in X_\log$ mapping to the images of $s$ and $x$ in $S_\an$ and $X_\an$ respectively.

Starting from the exact sequence
$$
\xymatrix{
1\ar[r] & \pi_1((X_\log)_{s'},{x'})\ar[r]  & \pi_1(X_\log,x')\ar[r] & \pi_1(S_\log,s')\ar[r] & 1,
}
$$
we have that the induced sequence of profinite completions
$$
\xymatrix{
1\ar[r] & \widehat{\pi}_1((X_\log)_{s'},{x'})\ar[r]  & \widehat{\pi}_1(X_\log,x')\ar[r] & \widehat{\pi}_1(S_\log,s')\ar[r] & 1
}
$$
is still exact, thanks to Proposition 4 of \cite{profinite}: this uses the fact that all fundamental groups are finitely generated (which follows from standard results about fundamental groups of triangulable spaces) and therefore are ``type F'', and Proposition 1 of the same paper, ensuring that the action $\pi_1(X_\log,x')\to \mathrm{Aut}(\widehat{\pi}_1((X_\log)_{s'},{x'}))$ induced by conjugation factors through the profinite completion $\widehat{\pi}_1(X_\log,x')$. Now it suffices to note that we have compatible isomorphisms $\widehat{\pi}_1(X_\log,x')\cong \pi_1^\ket(X,x)$ and $ \widehat{\pi}_1(S_\log,s')\cong \pi_1^\ket(S,s)$ (thanks to \cite[Theorem 6.4]{KNIRS}), inducing then an isomorphism $\widehat{\pi}_1((X_\log)_{s'},{x'})\cong \pi_1^\ket(X_s,x)$ as claimed.
\end{remark}

\begin{remark}[K\"{u}nneth formula for the log fundamental group]
As in the case of schemes, from Theorem \ref{theorem:main} one can easily deduce a K\"{u}nneth type formula for the log fundamental group (see also \cite[Proposition 2.4]{hoshi}).

Let $X$ and $Y$ be connected fs log schemes over an algebraically closed field $k$, and assume that $X$ is proper, finitely presented and log smooth over $k$ (here $\Spec k$ is equipped with the trivial log structure). Let $x\to X$, $y\to Y$ be strict log geometric points with residue field $k$. Note that we can choose a log geometric point $z\to X\times_k Y$ mapping to $x$ and $y$ by taking a log geometric point of the fibered product $x\times_k y$. The projections $X\times_k Y\to X$ and $X\times_k Y\to Y$ induce a homomorphism $\pi_1^\ket(X\times_k Y,z)\to \pi_1^\ket(X,x)\times \pi_1^\ket(Y,y)$.

This is an isomorphism, as one can easily show by applying Theorem \ref{theorem:main} to the projection $X\times_k Y\to Y$, and comparing the resulting short exact sequence with the split sequence 
$$
\xymatrix{
1\ar[r] & \pi_1^\ket(X,{x})\ar[r] & \pi_1^\ket(X,x)\times \pi_1^\ket(Y,y)\ar[r] & \pi_1^\ket(Y,{y})\ar[r] & 1,
}
$$
taking into account the fact that the homomorphism $\pi_1^\ket(X_y,z)\cong \pi_1^\ket(X,x)$ induced by the map $X_y\to X$ is an isomorphism (here $X_y$ is the fiber of $X\times_k Y\to Y$ over $y\to Y$, which of course coincides with the fiber of $X\to \Spec k$ over $y\to \Spec k$). This follows for example by applying Theorem \ref{thm:left.exact} to the morphism $X\to \Spec k$ with respect to the log geometric point $y\to \Spec k$.
\end{remark}

\bibliographystyle{alpha}
\bibliography{biblio}

\end{document}